\documentclass[final,twocolumn]{elsarticle}

\usepackage[utf8]{inputenc}
\usepackage[english]{babel}

\usepackage{amsmath}
\usepackage{bm} 

\usepackage{amssymb}
\usepackage{bbm}

\let\theoremstyle\relax
\usepackage{amsthm}

\usepackage{mathtools}

\usepackage{todonotes}

\theoremstyle{definition}
\newtheorem{definition}{Definition}
\newtheorem{theorem}{Theorem}
\newtheorem{lemma}{Lemma}
\newtheorem{assumption}{Assumption}
\newtheorem{problem}{Problem}
\newtheorem{remark}{Remark}

\usepackage{hyperref}
\hypersetup{
    colorlinks = true,
    citecolor = blue
}

\usepackage{graphicx}
\graphicspath{ {resources/} }
\usepackage{booktabs}
\usepackage{float}

\journal{System \& Control Letters}

\begin{document}

\begin{frontmatter}

\title{An Energy-Optimal Framework for Assignment and Trajectory Generation in Teams of Autonomous Agents}
\author[Delaware]{Logan Beaver}\ead{lebeaver@udel.edu}
\author[Delaware]{Andreas Malikopoulos}\ead{andreas@udel.edu}

\address[Delaware]{Department of Mechanical Engineering, University of Delaware, Newark, Delaware, 19716, USA}


\begin{abstract}
    In this paper, we present an approach for solving the problem of moving $N$ homogeneous agents into $M \geq N$ goal locations along energy-minimizing trajectories. We propose a decentralized framework that only requires knowledge of the goal locations and partial observations of the global state by each agent. The framework includes guarantees on safety through dynamic constraints, and a method to impose a dynamic, global priority ordering on the agents. A solution to the goal assignment and trajectory generation problems are derived in the form of a binary program and a nonlinear system of equations. Then, we present the conditions for optimality and characterize the conditions under which our algorithm is guaranteed to converge to a unique assignment of agents to goals. We also solve the fully constrained decentralized trajectory generation problem considering the state, control, and safety constraints. Finally, we validate the efficacy of our approach through a numerical simulation in MATLAB. 
\end{abstract}

\end{frontmatter}

\section{INTRODUCTION}
\subsection{Motivation}
Complex systems consist of diverse entities that interact both in space and time \cite{Malikopoulos2016c}. Referring to something as complex implies that it consists of interdependent entities or agents that can adapt, i.e., they can respond to their local and global environment. Complex systems appear in many applications, including cooperation between components of autonomous systems, sensor fusion, and natural biological systems. As we move to increasingly complex systems \cite{Malikopoulos2015},
new control approaches are needed to optimize the impact of individuals on system-level
behavior through the control of individual entities \cite{Malikopoulos2015b,Malikopoulos}.

Swarms are typical complex system which have attracted considerable attention in many applications, e.g., transportation, exploration, construction, surveillance, and manufacturing. As discussed by Oh et al. \cite{Oh2017}, swarms are especially attractive due to their natural parallelization and general adaptability. One of the typical multiagent applications is creating desired formations. However, due to cost constraints on any real swarm of autonomous agents, e.g., limited computation capabilities, battery capacity, and sensing capabilities, any efficient control approach needs to take into account the energy consumption of each agent. Moving agents into a desired formation has been explored previously; however, creating this formation while minimizing energy consumption is an open problem.

\subsection{Related Work}

Brambilla et al. \cite{Brambilla2013} classified approaches to swarms into two distinct groups, macroscopic, and microscopic.
Macroscopic approaches generate group behavior from a system of partial differential equations which are spatially discretized and applied to individual agents; this approach is fundamentally based on work by Turing \cite{Turing1952}, and is used extensively in bio-inspired formation and pattern forming \cite{Oh2015}.
Our approach is microscopic; that is, we control the behavior of individual agents to achieve some desired global outcome.
Microscopic approaches are based on the seminal work by Reynolds \cite{Reynolds1987}, which applied an agent-based method to capture the flocking behavior of birds.

There is a rich literature on the creation of a desired formation, such as the construction of rigid formations from triangular sub-structures \cite{Guo2010,Hanada2007}, formation algorithms inspired by crystal growth \cite{Song}, and growing swarm formations in a lattice structure \cite{Lee2008}. It is also possible to build formations using only scalar, bearing, or distance measurements to move agents into a desired formation \cite{Swartling2014,Lin2004}. Olfati-Saber et al. \cite{Olfati-Saber2007} proved that many formation problems may be solved by applying a modified form of the basic consensus algorithm. However, none of these approaches consider the energy cost to individual agents in the swarm.

A significant amount of work has applied optimization methods to designing potential fields for agent interaction \cite{Wang2013,Sun,Xu2015,Rajasree,Vasarhelyi2018OptimizedEnvironmentsb}. However, these approaches optimize the shape of the potential field and do not consider the energy consumption of individual agents. Previous work by Turpin et al. \cite{Turpin} has generated optimal assignments using a centralized planner. Other approaches require global information about the system  (e.g., \cite{Turpin} requires globally unique assignments a priori, \cite{Morgan2016} requires the graph diameter of the system, and \cite{Rubenstein2012} imposes global ``seed'' agents on the swarm). 

Our approach is decentralized, and thus, each agent may only partially observe the entire system state. 
The latter results in a non-classical information structure and many techniques for solving centralized systems do not hold \cite{Dave2018StructuralCommunication}. To address this problem, one may impose a priority ordering on the agents. This has been achieved in previous work through a centralized controller \cite{Turpin2013ConcurrentRobots,Chalaki2019AnIntersectionsb}. In general, finding an optimal ordering is NP-Hard, and an optimal ordering is not always guaranteed to exist \cite{Ma2019SearchingFinding}. To reduce the complexity of ordering agents, much work has been done to decentralize the ordering problem, including applying discretized path-based heuristics \cite{Wu2019Multi-RobotProspects}, and reinforcement learning \cite{Sartoretti2018PRIMAL:Learning}. In contrast, our approach introduces a decentralized method of dynamically ordering agents that is path independent and relies only on information directly observable by each agent.

Our approach only requires agents to make partial observations of the entire system. This may lead to performance degradation relative to a centralized controller with global knowledge. However, this is a fundamental feature of decentralized control problems in general \cite{Witsenhausen1968AControl}. Other efforts have attempted to circumvent this issue with information sharing, either directly \cite{Dave2018StructuralCommunication} or through decentralized auctioning \cite{Morgan2016}. However, these approaches tend to require knowledge of the global communication graph, long delays before decisions are made, or both. In contrast, we embrace partial observability of the system and exploit it to reduce the computational load on each individual agent.

The main contributions of this paper are: (1) a decentralized set of \emph{interaction dynamics}, which impose a priority order on agents in a  decentralized manner, (2) an assignment algorithm that exploits the unconstrained optimal trajectory of the agents,  (3) guarantees on the stability of the proposed control policy, and (4) optimality conditions for the fully constrained collision avoidance problem in $\mathbb{R}^2$, as well as a locally optimal solution in the form of nonlinear algebraic equations.

\subsection{Organization of This Paper}
The remainder of this paper proceeds as follows. In Section \ref{sec:problem}, we formulate the decentralized optimal control problem, and we decompose it into the coupled assignment and trajectory generation subproblems.
In Section \ref{sec:assignment}, we present the conditions which guarantee system convergence along with the assignment problem.
Then, in Section \ref{sec:trajectory}, we prove that these conditions are satisfied by our framework and solve the optimal trajectory generation problem.
Finally, in Section \ref{sec:simulation}, we present a series of MATLAB simulations to show the performance of the algorithm, and we presented concluding remarks in Section \ref{sec:conclusion}.

\section{PROBLEM FORMULATION} \label{sec:problem}

We consider a swarm of $N\in\mathbb{N}$ autonomous agents indexed by the set $\mathcal{A} = \{1, \dots, N\}$. Our objective is to design a decentralized control framework to move the $N$ agents into $M\in\mathbb{N}$ goal positions, indexed by the set $\mathcal{F} = \{1, \dots M\}$. We consider the case where $N \leq M$, i.e., no redundant agents are brought to fill the formation. This requirement can be relaxed by defining a behavior for excess agents, such as idling \cite{Turpin2014}. Each agent $i\in\mathcal{A}$ follows double-integrator dynamics,
\begin{align}
\mathbf{\dot{p}}_i(t) &= \mathbf{v}_i(t) \label{eqn:pDynamics}, \\
\mathbf{\dot{v}}_i(t) &= \mathbf{u}_i(t) \label{eqn:vDynamics},
\end{align}
where $\mathbf{p}_i(t)\in\mathbb{R}^2$ and $\mathbf{v}_i(t)\in\mathbb{R}^2$ are the time-varying position and velocity vectors respectively, and $\mathbf{u}_i(t) \in \mathbb{R}^2$ is the control input (acceleration/deceleration) over time $t\in[t_i^0, t_i^f]$, where $t_i^0$ is the initial time for agent $i$ and $t_i^f\in \mathbb{R}_{>0}$ is the terminal time for agent $i$. Additionally, each agent's control input and velocity are bounded by
\begin{align}
	v_{\min} \leq ||\mathbf{v}_i(t)|| \leq v_{\max},
	\label{eqn:vBounds}\\
	u_{\min} \leq ||\mathbf{u}_i(t)|| \leq u_{\max},
	\label{eqn:uBounds}
\end{align}
where $||\cdot||$ is the Euclidean norm. Thus, the state of each agent $i\in\mathcal{A}$ is given by the time-varying vector
\begin{equation}
\mathbf{x}_i(t) =
\begin{bmatrix}
\mathbf{p}_i(t) \\ \mathbf{v}_i(t)
\end{bmatrix},
\end{equation}
and we denote the global (system) state as
\begin{equation}  \label{eq:globalState}
    \mathbf{x}(t) =
    \begin{bmatrix} 
    \mathbf{x}_1(t) \\
    \dots \\
    \mathbf{x}_N(t)
    \end{bmatrix}.
\end{equation}

The energy consumption of any agent $i\in\mathcal{A}$ is given by
\begin{equation} \label{eqn:energy}
\dot{E}_i(t) = \frac{1}{2}||\mathbf{u}_i(t)||^2.
\end{equation}
We select the $L^2$ norm of the control input as our energy model since, in general, acceleration/deceleration requires more energy than applying no control input. Therefore, we expect that minimizing the acceleration/deceleration of each agent will yield a proportional reduction in energy consumption.

Our objective is to develop a decentralized framework for the $N$ agents to optimally, in terms of energy, create any desired formation of $M$ points while avoiding collisions between agents.
\begin{definition} \label{def:goals}
	The \textit{desired formation} is the set of time-varying vectors $\mathcal{G} = \{\mathbf{p}_j^*(t) \in \mathbb{R}^2 ~ | ~ j \in \mathcal{F}\}$. The set $\mathcal{G}$ can be prescribed offline, i.e., by a human  designer, or online by a high-level planner.
\end{definition}

In this framework, the agents are cooperative and capable of communication within a neighborhood, which we define next.

\begin{definition} \label{def:neighborhood}
	The \textit{neighborhood} of agent $i\in\mathcal{A}$ is the time-varying set 
	\begin{equation*}
	\mathcal{N}_i(t) = \Big\{j\in\mathcal{A} ~ \Big| ~ \big|\big|\textbf{p}_i(t) - \textbf{p}_j(t) \big|\big| \leq h\Big\},
	\end{equation*}
	where $h\in\mathbb{R}$ is the sensing and communication horizon of each agent.
\end{definition}

Agent $i\in\mathcal{A}$ is also able to measure the relative position of any neighboring agent $j\in\mathcal{N}_i(t)$, i.e., agent $i$ makes partial observations of the global state. We denote the relative position between two agents $i$ and $j$ by the vector
\begin{align} \label{eq:s}
    \mathbf{s}_{ij}(t) = \mathbf{p}_j(t) - \mathbf{p}_i(t).
\end{align}
Each agent $i\in\mathcal{A}$ occupies a closed disk of radius $R$; hence, to guarantee safety for agent $i$ we impose the following constraints for all agents $i\in\mathcal{A},~~j\in\mathcal{N}_i(t), ~ j\neq i$,
\begin{align} \label{eqn:collisionCondition}
||\mathbf{s}_{ij}(t)|| &\geq  2R, \quad \forall t\in[t_i^0, t_i^f], \\
h &>>  2R. \label{eqn:sensingCondition}
\end{align}
Condition \eqref{eqn:collisionCondition} is our safety constraint, which ensures that no two agents collide. We also impose the strict form of \eqref{eqn:collisionCondition} pairwise to all goals in the desired formation, $\mathcal{G}$. Equation \eqref{eqn:sensingCondition} is a system-level constraint which ensures agents are able to detect each other prior to a collision.

In our modeling framework we impose the following assumptions:
\begin{assumption} \label{smp:perfect}
	The state $\mathbf{x}_i(t)$ for each agent $i\in\mathcal{A}$ is perfectly observed and there is negligible communication delay between the agents. 
\end{assumption}

Assumption \ref{smp:perfect} is required to evaluate the idealized performance of the generated optimal solution. In general, this assumption may be relaxed by formulating a stochastic optimal control problem to generate agent trajectories.

\begin{assumption} \label{smp:energy}
	The energy cost of communication is negligible; the only energy consumption is in the form of (\ref{eqn:energy}). 
\end{assumption}

The strength of this assumption is application dependent. For cases with long-distance communications or high data rates, the trade-off between communication and motion costs can be controlled by varying the sensing and communicating radius, $h$, of the agents.

To solve the desired formation problem, we first relax the inter-agent collision avoidance constraint to decouple the agent trajectories. This decoupling reduces the problem from a single mixed-integer program to a coupled pair of binary and quadratic programs, which we solve sequentially. This decoupling is common in the literature \cite{Turpin,Morgan2016}, and usually does not affect the outcome of the assignment problem.

Next, we present some preliminary results before decomposing the desired formation problem into the two subproblems, minimum-energy goal assignment and trajectory generation.

\subsection{Preliminaries}

First we consider that any agent $i\in\mathcal{A}$ obeys double-integrator dynamics, \eqref{eqn:pDynamics} - \eqref{eqn:vDynamics}, and has an energy model with the form of \eqref{eqn:energy}.Then, we let the state and control trajectories of agent $i$ be unconstrained, i.e., relax \eqref{eqn:vBounds}, \eqref{eqn:uBounds}, and \eqref{eqn:collisionCondition}. In this case, if $i$ is traveling between two fixed states, the unconstrained minimum-energy trajectory is given by the following system of linear equations:
\begin{align}
    \mathbf{u}_i(t) &= \mathbf{a}_i~t + \mathbf{b}_i, \label{eq:uUnc}\\
    \mathbf{v}_i(t) &= \frac{\mathbf{a}_i}{2}~t^2 + \mathbf{b}_i~t + \mathbf{c}_i, \label{eq:vUnc}\\
    \mathbf{p}_i(t) &= \frac{\mathbf{a}_i}{6}~t^3 + \frac{\mathbf{b}_i}{2}~t^2 + \mathbf{c}_i~t + \mathbf{d}_i, \label{eq:pUnc}
\end{align}
where $\mathbf{a}_i, \mathbf{b}_i, \mathbf{c}_i, $ and $\mathbf{d}_i$ are constant vectors of integration.
The derivation of \eqref{eq:uUnc} - \eqref{eq:pUnc} is straightforward and can be found in \cite{Malikopoulos2018}.

Thus, the energy consumed for any unconstrained trajectory of agent $i\in\mathcal{A}$ at time $t$ traveling towards the goal $j\in\mathcal{F}$ is given by
\begin{align}
    E_i^j(t) &= \int_{t}^{t_i^f} ||\mathbf{u}_i(\tau)||^2 ~ d\tau = (t_f^3-t^3)\Big(\frac{a_{i,x}^2 + a_{i,y}^2}{3}\Big)\notag \\ &~+ (t_f^2-t^2)\Big(a_{i,x}~b_{i,x}+a_{i,y}~b_{i,y}\Big) \notag\\
    &~+ (t_f - t)\Big(\frac{b_{i,x}^2 + b_{i,y}^2}{2}\Big), \label{eq:uncEnergy}
\end{align}
where $t\in[t_i^0, t_i^f]$, and $\mathbf{a}_i = [a_{i,x}, a_{i,y}]^T$, $\mathbf{b}_i = [b_{i,x}, b_{i,y}]^T$
are the coefficients of \eqref{eq:uUnc}.

Next, we present the interaction dynamics between agents. To resolve any conflict between agents, we consider the following objectively measurable constants: 1) neighborhood size, 2) energy required to reach a goal, and 3) agent index, which may be arbitrarily assigned.
Each of these quantities has an associated indicator function for comparing two agents $i,j\in\mathcal{A}, ~ j\neq i$,
\begin{align}
    \mathbbm{1}^\mathcal{N}_{ij}(t) &\coloneqq \label{eq:ind1}
    \begin{cases}
    1 & |\mathcal{N}_i(t)| > |\mathcal{N}_j(t)|,\\
    0 & |\mathcal{N}_i(t)| \leq |\mathcal{N}_j(t)|,
    \end{cases} \\
    \mathbbm{1}^\mathcal{E}_{ij}(t) &\coloneqq \label{eq:ind2}
    \begin{cases}
    1 & E_i(t) > E_j(t),\\
    0 & E_i(t) \leq E_j(t),
    \end{cases}
    \\
    \mathbbm{1}^\mathcal{A}_{ij}(t) &\coloneqq \label{eq:ind3}
    \begin{cases}
    1 & i > j,\\
    0 & i < j.
    \end{cases}
\end{align}
Next, we define the interaction dynamics by combining \eqref{eq:ind1} - \eqref{eq:ind3} into a single indicator function.
\begin{definition} \label{df:interactionDynamics}
    We define the \emph{interaction dynamics} between any agent $i\in\mathcal{A}$ and another agent $j\in\mathcal{N}_i(t), j\neq i$ as
    \begin{align}
        \mathbbm{1}_{ij}^C(t) &= \mathbbm{1}^{\mathcal{N}}_{ij}(t) 
        + \big(1 - \mathbbm{1}^{\mathcal{N}}_{ij}(t)\big)\big(1 - \mathbbm{1}^{\mathcal{N}}_{ji}(t)\big) \notag\\ &\Big( \mathbbm{1}^{\mathcal{E}}_{ij}(t) 
          + \big(1 - \mathbbm{1}^{\mathcal{E}}_{ij}(t)\big)\big(1 - \mathbbm{1}^{\mathcal{E}}_{ji}(t)\big) \mathbbm{1}^{\mathcal{A}}_{ij}(t) \Big), \label{eq:interactionDynamics}
    \end{align}
    where $\mathbbm{1}^C_{ij} = 1$ implies agent $i$ has priority over agent $j$, and $\mathbbm{1}^C_{ij}=0$ implies that agent $j$ has priority over agent $i$.
\end{definition}

The interaction dynamics are instantaneously and noiselessly measured and communicated by each agent under Assumption \ref{smp:perfect}. Whenever two agents have a conflict (i.e., share an assigned goal, or have overlapping assignments) \eqref{eq:interactionDynamics} is used to impose an order on the agents such that higher priority agents act first.

\begin{remark}\label{rmk:unambiguous}
For any pair of agents $i\in\mathcal{A}, j\in\mathcal{N}_i(t), j\neq i$, it is always true that $\mathbbm{1}_{ij}^C(t) = 1 -\mathbbm{1}_{ji}^C(t)$, i.e., the outcome of the interaction dynamics \eqref{eq:interactionDynamics} is always unambiguous, and therefore it imposes an order on any pair of agents.
\end{remark}

Remark \ref{rmk:unambiguous} can be proven by simply enumerating all cases of \eqref{eq:ind1} - \eqref{eq:ind3}.

\section{Optimal Goal Assignment} \label{sec:assignment}

The optimal solution of the assignment problem must assign each agent to a goal such that the total unconstrained energy cost, given by \eqref{eq:uncEnergy}, is minimized.
In our framework, each agent $i\in\mathcal{A}$ only has information about the positions of its neighbors, $j\in\mathcal{N}_i(t)$, and the goal positions prescribed by $\mathcal{G}$.
Agent $i$ derives the goal assignment using a binary matrix $\mathbf{A}_i$(t), which we define next.
\begin{definition}\label{def:assignmentMatrix}
    For each agent $i\in\mathcal{A}$, we define the \emph{assignment matrix}, $\mathbf{A}_i(t)$, as an $|\mathcal{N}_i(t)|\times M$ matrix with binary elements. The elements of $\mathbf{A}_i(t)$ map each agent to exactly one goal, and each goal to no more than one agent.
\end{definition}

The assignment matrix for agent $i\in\mathcal{A}$ assigns all agents in $\mathcal{N}_i(t)$ to goals by considering the cost \eqref{eq:uncEnergy}. We discuss the details of the optimal assignment problem later in this section.

Next we define the prescribed goal, which determines how each agent $i\in\mathcal{A}$ assigns itself a goal.
\begin{definition} \label{def:prescribedGoal}
	We define the \emph{prescribed goal} for agent $i\in\mathcal{A}$ as the goal assigned to agent $i$ by the rule,
	\begin{equation}
		\mathbf{p}_i^a(t) \in \big\{\mathbf{p}_k \in \mathcal{G} ~|~ a_{ik}=1,~ a_{ik} \in\mathbf{A}_i(t), k\in\mathcal{F} \big\},
	\end{equation}
	where $\mathbf{A}_i(t)$ is the assignment matrix, and the right hand side is a singleton set, i.e., agent $i$ is assigned to exactly one goal.
\end{definition}

Next, we present the goal assignment algorithm in terms of some agent $i\in\mathcal{A}$. However, as this framework is cooperative, each step is performed by all individuals simultaneously.

In some cases, multiple agents may select the same prescribed goal. This may occur when two agents $i\in\mathcal{A}, j\in\mathcal{N}_i(t), j\neq i$ have different neighborhoods and use conflicting information to solve their local assignment problem. 
This motivates the introduction of competing agents, which we define next.
\begin{definition}\label{def:competingAgents}
	For agent $i\in\mathcal{A}$, we define the set of \textit{competing agents} as
	\begin{equation*}
		\mathcal{C}_i(t) = \Big\{k\in\mathcal{N}_i(t) ~|~ \mathbf{p}_k^a(t) = \mathbf{p}_i^a(t) \Big\}.
	\end{equation*}
\end{definition}
When $\big|\mathcal{C}_i\big| > 1$ there are at least two agents, $i,j\in\mathcal{N}_i(t)$ which are assigned to the same goal.
In this case, all but one agent in $\mathcal{C}_i(t)$ must be \emph{permanently banned} from the goal $\mathbf{p}_i^a(t)$. Next, we define the banned goal set.
\begin{definition}\label{def:bannedSet}
    The \emph{banned goal set} for agent $i\in\mathcal{A}$ is defined as
    \begin{align} \label{eqn:bannedUpdate}
        \mathcal{B}_i(t) &= \Big\{ g\in\mathcal{F} ~\Big|~ \mathbf{p}_i^a(\tau) = \mathbf{p}_g(\tau)\in\mathcal{G}, \notag\\
        &\Big(\prod_{j\in\mathcal{C}_i(\tau),j\neq i} \mathbbm{1}_{ij}^c(\tau) \Big) = 0, ~ \exists ~\tau\in[t_i^0, t] 
    ~\Big\},
    \end{align}
    i.e., the set of all goals which agent $i\in\mathcal{A}$ had a conflict over and did not have priority per Definition \ref{df:interactionDynamics}.
\end{definition}

Banning is achieved by applying \eqref{eqn:bannedUpdate} to all agents $j\in\mathcal{C}_i(t)$ whenever $|\mathcal{C}_i(t)| > 1$. After the banning step is completed, agent $i\in\mathcal{A}$ checks if the size of $\mathcal{B}_i(t)$ has increased. If so, agent $i$ increases the value of  $t_i^f$ by
\begin{equation}
    t_i^f = t + T,
\end{equation}
where $t$ is the current time, and $T\in\mathbb{R}_{>0}$ is a system parameter. This allows agent $i$ a sufficient amount of time to reach its new goal.

Next, each agent broadcasts its new set of banned goals to all of its neighbors. Any agent who was banned from $\mathcal{C}_i(t)$ assigns itself to a new goal. However, this may cause new agents to enter $\mathcal{C}_i(t)$ as they are banned from other goals. To ensure each agent $j\in\mathcal{N}_i(t)$ is assigned to a unique goal, the assignment and banning steps are iterated until the condition
\begin{equation} \label{eq:conflictCondition}
    \big|\mathcal{C}_j(t)\big| = 1, \quad \forall j\in\mathcal{N}_i(t),
\end{equation}
is satisfied. For a given neighborhood $\mathcal{N}_i(t), i\in\mathcal{A}$, some number of agents will be assigned to the goal $g\in\mathcal{F}$. After the first banning step, all agents except the one which was assigned to goal $g$ are permanently banned and may never be assigned to it again. If additional agents are assigned to $g$, then this process will repeat for at most $N-1$ iterations. Afterwards every goal $g$ will have at most one agent from $\mathcal{N}_i(t)$ assigned to it. Thus, we will have $|\mathcal{C}_j(t)| = 1$ for all $j\in\mathcal{N}_i(t)$ for every $i\in\mathcal{A}$.

We enforce the banned goals through a constraint on the assignment problem, which follows.
\begin{problem}[Goal Assignment] \label{prb:assignment}
	Each agent $i\in\mathcal{A}$ selects its prescribed goal (Definition \ref{def:prescribedGoal}) by solving the following binary program
	\begin{align}
	\underset{a_{jk}\in\mathbf{A}_i}{\text{min}} \Bigg\{
	\sum_{j\in\mathcal{N}_i(t)} \sum_{k\in\mathcal{F}} a_{jk} E^k_j(t) \Bigg\}
	\end{align}
	subject to:
	\begin{align}
			 \sum_{j\in\mathcal{F}} a_{jk} &= 1, \quad k\in\mathcal{N}_i(t), \label{eqn:p11} \\
			 \sum_{k\in\mathcal{N}_i(t)} a_{jk} &\leq 1, \quad j\in\mathcal{F},\label{eqn:p12}\\
			 a_{jk} &= 0, \quad \forall~j\in\mathcal{B}_k(t),~ k\in\mathcal{N}_i(t),\label{eqn:p13} \\
			 a_{jk} &\in \{0, 1\} \notag.
	\end{align}
	This process is repeated by each agent, $i\in\mathcal{A}$, until \eqref{eq:conflictCondition} is satisfied for all $j\in\mathcal{N}_i(t)$.
\end{problem}

As the conflict condition in Problem \ref{prb:assignment} explicitly depends on the neighborhood of agent $i\in\mathcal{A}$, Problem \ref{prb:assignment} may need to be recalculated each time the neighborhood of agent $i$ switches. The assignments generated by Problem \ref{prb:assignment} are guaranteed to bring each agent to a unique goal; we show this with the help of Assumption \ref{smp:bans} and Lemma \ref{lma:solutionExistance}.

\begin{assumption}\label{smp:bans}
For every agent $i\in\mathcal{A}$, for all $t\in[t_i^0, t_i^f]$, the inequality $\big|\big(\mathcal{F} \setminus \bigcup_{j\in\mathcal{N}_i(t)} \mathcal{B}_j(t)\big) 
\geq |\mathcal{N}_i(t)|$ holds.
\end{assumption}

Assumption \ref{smp:bans} is a condition that is sufficient but not necessary to prove convergence of our proposed optimal controller.
Intuitively, Assumption \ref{smp:bans} only requires that one agent does not ban many agents from many goals.
Due to the minimum-energy nature of our framework, this scenario is unlikely; additionally, permanent banning may be relaxed to temporary banning in a way that Assumption \ref{smp:bans} is always satisfied.

\begin{lemma}[Solution Existence] \label{lma:solutionExistance}
	Under Assumption \ref{smp:bans}, the feasible region of Problem \ref{prb:assignment} is nonempty for agent $i$.
\end{lemma}

\begin{proof}
    Let the set of goals available to all agents in the neighborhood of agent $i\in\mathcal{A}$ be denoted by the set
    \begin{equation} \label{eqn:feasibleGoals}
        \mathcal{L}_i(t) = \{p\in\mathcal{F} ~|~ p\not\in\mathcal{B}_j(t), ~ \forall j\in\mathcal{N}_i(t) \}.
    \end{equation}
    Let the injective function $w : \mathcal{A} \to \mathcal{F}$ map each agent to a goal. 
    By Assumption \ref{smp:bans}, $|\mathcal{N}_i(t)| \leq |\mathcal{L}_i(t)|$, thus a function $w$ exists.
    As $w$ is injective, the imposed mapping must satisfy (\ref{eqn:p11}) and (\ref{eqn:p12}). Likewise, $\mathcal{L}_i(t)\bigcap\mathcal{B}_j(t) = \emptyset$ for all $j\in\mathcal{N}_i(t)$. Thus, $w$ must satisfy \eqref{eqn:p13}. Therefore, the mapping imposed by the function $w$ is a feasible solution to Problem \ref{prb:assignment}.
\end{proof}
Next, we show that for a sufficiently large value of $T$ the convergence of all agents to goals is guaranteed.

\begin{theorem}[Assignment Convergence] \label{thm:assignmentConvergence}
	Under Assumption \ref{smp:bans}, for sufficiently large values of the initial $t_i^f$ and $T$, and if the energy-optimal trajectories for agent $i\in\mathcal{A}$ never increase the unconstrained energy cost \eqref{eq:uncEnergy}, then $t_i^f$ must have an upper bound for all $i\in\mathcal{A}$.
\end{theorem}

\begin{proof}
    Let $\{g_n\}_{n\in\mathbb{N}}$ be the sequence of goals assigned to agent $i\in\mathcal{A}$ by the solution of Problem \ref{prb:assignment}.  By Lemma \ref{lma:solutionExistance}, $\{g_n\}_{n\in\mathbb{N}}$ must not be empty, and the elements of this sequence are natural numbers bounded by $1\leq g_n \leq M$. Thus, the range of this sequence is compact, and the sequence must be either (1) finite, or (2) convergent, or (3) periodic.
    
    $1$) For a finite sequence there is nothing to prove, as $t_i^f$ is upper bounded by $t_{i,\text{initial}}^f + MT$.
    
    $2$) Under the discrete metric, an infinite convergent sequence requires that there exists $\
    N\in\mathbb{N}_{>0}$ such that $g_n = p$ for all $n>N$ for some formation index $p\in\mathcal{F}$. This reduces to case $1$, as $t_i^f$ does not increase for repeated assignments to the same goal.
    
    $3$) By the Bolzano-Weierstrass Theorem, an infinite non-convergent sequence $\{g_n\}_{n\in\mathbb{N}}$ must have a convergent subsequence, i.e., agent $i$ is assigned to some subset of goals $\mathcal{I}\subseteq\mathcal{F}$
    infinitely many times with some constant number of intermediate assignments, $P_{g}$, for each goal $g\in\mathcal{I}$.
    Necessarily, $\mathcal{I}\bigcap\mathcal{B}_i(t) = \emptyset$ for all $t\in[t_i^0,t_i^f]$ from the construction of the banned goal set.
    This implies that, by the update method of $t_i^f$, the position of all goals, $g\in\mathcal{I}$ must only be considered at time $t_i^f$.
    
    This implies that the goals available to agent $i$, i.e., $\mathcal{I} = \mathcal{F}\setminus\mathcal{B}_i(t_i^f)$, must be shared between $n>0$ other periodic agents. Hence, at some time $t_1$  a goal, $g\in\mathcal{I}$, must be an optimal assignment for agent $i$, a non optimal assignment at time $t_2>t_1$ and an optimal assignment again at time $t_3$ which corresponds to the ${P_g}^{th}$ assignment.
    As $t_3>t_2>t_1$, the energy required to move agent $i$ to goal $g$ satisfies
    \begin{align}
        E_i^g(t_1) &\leq E_i^k(t_1), \\
        E_i^k(t_2) &\leq E_i^g(t_2), \\
        E_i^g(t_3) &\leq E_i^k(t_3),
    \end{align}
    for any other goal $k\in\mathcal{I}, k\neq g$. Therefore, for agent $i$ to follow an energy optimal trajectory under our premise, it must never increase the energy required to reach is assigned goal, which implies
    \begin{align}
        E_i^g(t_1) &\geq E_i^g(t_2), \\
        E_i^k(t_2) &\geq E_i^k(t_3),
    \end{align}
    this implies
    \begin{equation}
        E_i^k(t_1) \geq E_i^k(t_3),
    \end{equation}
    which
    is only possible if agent $i$ simultaneously approaches all goals $k\in\mathcal{I}$. This implies that goals $g$ and $k$ are arbitrarily close, which violates the minimum spacing requirements of the goals; therefore no such periodic behavior may exist.
\end{proof}

Note that Theorem \ref{thm:assignmentConvergence} bounds the arrival time of agent $i\in\mathcal{A}$ to any goal $g\in\mathcal{F}$. A similar bound may be found for the total energy consumed, i.e.,
\begin{equation*}
    E_i^g(t) \leq \frac{1}{2}(t_{i,\text{initial}}^f + MT)\cdot\max\big\{|u_{\min}|, |u_{\max}|\big\}^2.
\end{equation*}

Next, we formulate the optimal trajectory generation problem for each agent and prove that the resulting trajectories always satisfy the premise of Theorem \ref{thm:assignmentConvergence}.

\section{Optimal Trajectory Generation} \label{sec:trajectory}

After the goal assignment is complete, each agent must generate a collision-free trajectory to their assigned goal.
The trajectories must minimize the agent's total energy cost subject to dynamic, boundary, and collision constraints.
The initial and final state constraints for each agent $i\in\mathcal{A}$ are given by
\begin{align}
    \mathbf{p}_i(t_i^0) &= \mathbf{p}_i^0, &\mathbf{v}_i(t_i^0) &= \mathbf{v}_i^0, \label{eqn:ICs}\\
	\mathbf{p}_i(t_i^f) &= \mathbf{p}_i^a(t_i^f), &\mathbf{v}_i(t_i^f) &= \dot{\mathbf{p}}_i^a(t_i^f), \label{eqn:BCs}
\end{align}
where the conditions at $t_i^f$ come from the solution of Problem \ref{prb:assignment}.

Whenever an agent must steer to avoid collisions, we will apply the agent interaction dynamics (Definition \ref{df:interactionDynamics}) to impose an order on the agents such that lower priority agents must steer to avoid the higher priority ones. Thus, we may simplify the collision avoidance constraint for agent $i\in\mathcal{A}$ to
\begin{align} \label{eqn:collisionOrdered}
    ||\mathbf{s}_{ij}(t)|| \geq 2R, ~&~ \forall\, j \in \{ k\in\mathcal{A} ~|~ \mathbbm{1}_{ik}(t) = 0  \}, \\
     &\forall\,t\in[t_i^0, t_i^f],\notag
\end{align}
which will always guarantee safety for all agents.

We may then formulate the decentralized optimal trajectory generation problem.
\begin{problem} \label{prb:trajectory}
	For each agent $i\in\mathcal{A}$, find the optimal control input, $\mathbf{u}_i(t)$, which minimizes the energy consumption of agent $i$ and satisfies its boundary conditions and safety constraints.
	\begin{align}	\label{eqn:hamiltonianOptimization}
	    &\min_{\mathbf{u}_i(t)}  \frac{1}{2} \int_{t_i^0}^{t_i^f} ||\textbf{u}_i(t)||^2 ~ dt \\
        &\text{subject to:} ~~ \eqref{eqn:pDynamics}, \eqref{eqn:vDynamics}, \eqref{eqn:ICs}, \eqref{eqn:BCs}, \text{ and } \eqref{eqn:collisionOrdered}. \notag
	    \end{align}
\end{problem}

By imposing an order on the agents, we can show that the solution of Problem \ref{prb:trajectory} will always satisfy the premise of Theorem \ref{thm:assignmentConvergence}. First, Lemma \ref{lma:unconstrainedEnergyReduced} shows that an unconstrained trajectory must never increase the energy required to reach a goal.

\begin{lemma} \label{lma:unconstrainedEnergyReduced}
    For any agent $i\in\mathcal{A}$, following the unconstrained trajectory, the energy cost \eqref{eq:uncEnergy} required to reach a fixed goal $g\in\mathcal{F}$ is not increasing.
\end{lemma}

\begin{proof}

    We may write the derivative of \eqref{eq:uncEnergy} along an unconstrained trajectory as
    \begin{align}
        \frac{dE_i^g(t)}{dt} &= \lim_{\delta\to0} \frac{1}{\delta}\Bigg( \int_{t+\delta}^{t_i^f} ||\mathbf{u}_i(\tau)||^2 d\tau - \int_{t}^{t_i^f} ||\mathbf{u}_i(\tau)||^2 d\tau \Bigg)\notag \\
        &= - \lim_{\delta\to0} \frac{1}{\delta}\int_{t}^{t+\delta}||\mathbf{u}_i(\tau)||^2 d\tau,
    \end{align}
    which is never positive. Therefore, \eqref{eq:uncEnergy} is never increasing.
\end{proof}

Next, we introduce Theorem \ref{thm:convergenceSatisfied}, which proves the premise of Theorem \ref{thm:assignmentConvergence} is always satisfied by any trajectory which is a feasible solution to Problem \ref{prb:trajectory}. 

\begin{theorem} \label{thm:convergenceSatisfied}
    If a solution to Problem \ref{prb:trajectory} exists for all agents, then Theorem \ref{thm:assignmentConvergence} is satisfied as long as Assumption \ref{smp:bans} holds.
\end{theorem}

\begin{proof}

    The case when any agent $i\in\mathcal{A}$ is moving with an unconstrained trajectory is covered by Lemma \ref{lma:unconstrainedEnergyReduced}, so we focus on the case when any of the safety constraints are active.

    Let $\mathcal{K}\subseteq\mathcal{A}$ be a group of agents which all have their safety constraint active over some interval $t\in[t_1, t_2]$. By Definition \ref{df:interactionDynamics}, there exists some $i\in\mathcal{K}$ such that $\mathbbm{1}_{ij}^C(t) = 1$ for all $j\in\mathcal{K},~j\neq i$. Therefore, agent $i$ satisfies Lemma \ref{lma:unconstrainedEnergyReduced} and always moves toward its assigned goal by Theorem \ref{thm:assignmentConvergence}.

    Next, consider agent $j\in\mathcal{V}\setminus\{i\}$ such that $\mathbbm{1}_{jk}^c(t) = 1$ for all $k\in\mathcal{K}\setminus\{i\}$. As agent $j$ may never be assigned to the same goal as $i$, there must exist some time $t_j < \min\{t_i^f, t_j^f\}$ such that $|\mathbf{s}_{ij}(t_j)| > 2R$ by the goal spacing rules. Thus, agent $j$ will move with an unconstrained trajectory for all $t\in[t_j, t_j^f]$. The above steps can be recursively applied until only a single agent remains, which follows an unconstrained trajectory for some finite time interval. This satisfies the premise of Theorem \ref{thm:assignmentConvergence}.
\end{proof}

\subsection{Hamiltonian Analysis}

We solve Problem \ref{prb:trajectory} by applying Hamiltonian analysis. 
We will follow the standard methodology used in optimal control theory for problems with interior point constraints. First, we start with the unconstrained solution, given by \eqref{eq:uUnc} - \eqref{eq:pUnc}. If this solution violates the state, control, or safety constraints, we piece it together with solutions corresponding to the violated constraint. These two arcs yield a set of algebraic equations that must be solved simultaneously using the boundary conditions \eqref{eqn:ICs} and \eqref{eqn:BCs} and interior conditions between the arcs. If the resulting trajectory, which includes the optimal switching time between the arcs, still violates any constraints, the new solution must be pieced together with a third arc corresponding to the new violated constraint. This process is repeated until no constraints are violated, which yields the energy-optimal state trajectory for each agent $i\in\mathcal{A}$.

The case where only the control and state constraints become active has been extensively studied in \cite{Malikopoulos2018}. Thus, we will relax the state and control constraints and only consider the safety constraint in this section. Next, we analyze the case where the collision avoidance constraint becomes active.

First, the safety constraint \eqref{eqn:collisionCondition} must be derived until the control input $\mathbf{u}_i(t)$ appears. To ensure smoothness in the derivatives we use the equivalent squared form of \eqref{eqn:collisionCondition}. This yields
\begin{equation} \label{eq:tangency}
	\mathbf{N}_i\big(t, \mathbf{x}_i(t)\big) = \begin{bmatrix}
	4R^2 - \mathbf{s}_{ij}(t)\cdot\mathbf{s}_{ij}(t)\\
	-\mathbf{s}_{ij}(t)\cdot\dot{\mathbf{s}}_{ij}(t)\\
	-\mathbf{s}_{ij}(t)\cdot\ddot{\mathbf{s}}_{ij}(t) - \dot{\mathbf{s}}_{ij}(t)
	\cdot\dot{\mathbf{s}}_{ij}(t)
	\end{bmatrix} \leq \mathbf{0},
\end{equation}
where the first two elements of $\mathbf{N}_i(t)$ are the \emph{tangency conditions} which must be satisfied at the start of a constrained arc, while the third element is augmented to the unconstrained Hamiltonian. The Hamiltonian is
\begin{align} 
	H_i = \frac{1}{2}||\mathbf{u}_i(t)||^2 + \boldsymbol{\lambda}_i^p(t)\cdot\mathbf{v}_i(t) + \boldsymbol{\lambda}_i^v(t)\cdot\mathbf{u}_i(t) \notag \\
	- \sum_{j\in\mathcal{N}_i} \mu_{ij}(t) \Big( \mathbf{s}_{ij}(t) \cdot \ddot{\mathbf{s}}_{ij}(t) + \dot{\mathbf{s}}_{ij}(t)\cdot\dot{\mathbf{s}}_{ij}(t)  \Big), \label{eq:hamiltonian}
\end{align}
where $\boldsymbol{\lambda}_i^p(t), \boldsymbol{\lambda}_i^v(t)$ are the position and velocity co-vectors, and $\mu_{ij}(t)$ is an inequality Lagrange multiplier with values
\begin{align}
	\mu_{ij}(t) = 
	\begin{cases}
		> 0  & \text{if  } {\mathbf{s}_{ij}(t) \cdot \ddot{\mathbf{s}}_{ij}(t) + \dot{\mathbf{s}}_{ij}(t)\cdot\dot{\mathbf{s}}_{ij}(t)} = 0, \\
		0  & \text{if  } {\mathbf{s}_{ij}(t) \cdot \ddot{\mathbf{s}}_{ij}(t) + \dot{\mathbf{s}}_{ij}(t)\cdot\dot{\mathbf{s}}_{ij}(t) > 0}.
	\end{cases}
\end{align}

To solve \eqref{eq:hamiltonian} for agent $i\in\mathcal{A}$, we consider two cases:
\begin{enumerate}
	\item all agents $j\in\mathcal{N}_i(t)$ satisfy $\mu_{ij} = 0$, and
	\item any agent $j\in\mathcal{N}_i(t)$ satisfies $\mu_{ij} > 0$.
\end{enumerate}

This results in a piecewise trajectory, where the first case corresponds to following an unconstrained trajectory, while the second corresponds to some collision avoidance constraints becoming active. Our task is to derive the form of the constrained arcs, then optimally piece the constrained and unconstrained arcs together. This will result in the optimal trajectory for agent $i$.

Next, we present the solution to the constrained case and the optimal time to transition between the two cases.

\subsection{Constrained Solution} \label{sec:constrainedSolution}

As with the assignment problem, the constrained solution is presented in terms of some agent $i\in\mathcal{A}$. However, the steps given here are performed simultaneously by all agents.
First, we define the conflict set.
\begin{definition} \label{def:conflictSet}
    We define the \emph{conflict set} for agent $i\in\mathcal{A}$ at time $t\in[t_i^0, t_i^f]$ as
    \begin{equation} \label{eq:conflict}
        \mathcal{V}_i(t) = \Big\{j\in\mathcal{N}_i(t) ~\big|~ \mu_{ij}(t) > 0, \mathbbm{1}_{ij}^c = 0 \Big\},
    \end{equation}
    i.e., the set of all agents which $i$ may collide with and have a higher priority than agent $i$.
\end{definition}

Agent $i$ must then steer to avoid all agents $j\in\mathcal{V}_i(t)$. To solve for the constrained trajectory we introduce Lemma \ref{lma:manyConstraints}, which considers the case when $|\mathcal{V}_i(t)| > 1.$

\begin{lemma}[Collision-Avoidance Constraints] \label{lma:manyConstraints}
    Let any agent $i\in\mathcal{A}$ be moving along a collision constrained arc such that $|\mathcal{V}_i(t)| > 1$. Then, only feasible trajectory for agent $i$ is to remain in contact with all agents $j\in\mathcal{V}_i(t)$ until agent $i$ exits this constrained arc. This unique trajectory is therefore optimal. 
\end{lemma}

\begin{proof}
    For any two agents $j,k\in\mathcal{V}_i(t)$ we have $||\mathbf{s}_{ij}(t)|| = 2R$ and $||\mathbf{s}_{ik}|| = 2R$. 
    This implies that the points $p_i(t)$, $p_j(t)$, and $p_k(t)$ must form an isosceles triangle with two edges of length $2R$ and base of length $2R \leq ||\mathbf{s}_{jk}(t)|| \leq 4R$.
    Therefore, the only feasible trajectory of agent $i$ is to maintain the isosceles triangle between $i, j,$, and $k$.
    As this is the only one feasible trajectory for agent $i$, it must be the optimal trajectory.
\end{proof}

Note that Lemma \ref{lma:manyConstraints} holds for a single constrained arc. As such, agent $i\in\mathcal{A}$ may exit to a different constrained arc with a new set $\mathcal{V}_i(t)$ or may exit to an unconstrained arc.

Next we consider the case where agent $i$ moves along a constrained arc with $|\mathcal{V}_i(t)| = 1$.
First, we use the 
Euler-Lagrange conditions to obtain,
\begin{align}
    \frac{\partial H_i}{\partial \mathbf{u}_i} &= 0, \label{eq:optimality}\\
    - \dot{\boldsymbol{\lambda}}_i &= \frac{\partial H_i}{\partial \mathbf{x}_i}. \label{eq:el}
\end{align}
Application of \eqref{eq:optimality} to \eqref{eq:hamiltonian} yields
\begin{align}
    \mathbf{u}_i(t) = -\boldsymbol{\lambda}_i^v(t) - \sum_{j\in\mathcal{V}_i(t)}\mu_{ij}(t)~\mathbf{s}_{ij}(t), \label{eq:uMany}
\end{align}
while \eqref{eq:el} results in
\begin{align}
    -\dot{\boldsymbol{\lambda}}_i^p(t) &= \sum_{j\in\mathcal{V}_i(t)}\mu_{ij}(t)~\ddot{\mathbf{s}}_{ij}(t),\label{eq:lpMany}\\
    -\dot{\boldsymbol{\lambda}}_i^v(t) &= \boldsymbol{\lambda}_i^p + \sum_{j\in\mathcal{V}_i(t)}\mu_{ij}(t)~\dot{\mathbf{s}}_{ij}(t).\label{eq:lvMany}
\end{align}

As $|\mathcal{V}_i(t)| = 1$, the optimality condition and Euler-Lagrange equations become
\begin{align}
    \mathbf{u}_i(t) &= -\boldsymbol{\lambda}_i^v(t) - \mu_{ij}(t)~\mathbf{s}_{ij}(t),\label{eq:u2}\\
    -\dot{\boldsymbol{\lambda}}_i^p(t) &= \mu_{ij}(t)~\ddot{\mathbf{s}}_{ij}(t),\label{eq:lpdot2}\\
    -\dot{\boldsymbol{\lambda}}_i^v(t) &= \boldsymbol{\lambda}_i^p + \mu_{ij}(t)~\dot{\mathbf{s}}_{ij}(t),\label{eq:lvdot2}
\end{align}
where $j\in\mathcal{V}_i(t)$.
We denote the relative speed between two agents $i$ and $j$ as
\begin{equation} \label{eq:relativeSpeed}
    a_{ij}(t) = ||\dot{\mathbf{s}}_{ij}(t)||.
\end{equation}

Next, we define a new orthonormal basis for $\mathbb{R}^2$.
\begin{definition}\label{def:localBasis}
    For an agent $i\in\mathcal{A}$ satisfying $\big|\mathcal{V}_i(t)\big|=1$, over some nonzero interval $t\in[t_1, t_2]$, where $a_{ij}(t)\neq0$, we define the \emph{contact basis} as
    \begin{align}
        \hat{p}_{ij}(t) = \frac{\mathbf{s}_{ij}(t)}{||\mathbf{s}_{ij}(t)||} = \frac{\mathbf{s}_{ij}(t)}{2R}, \label{eq:phat}\\
        \hat{q}_{ij}(t) = \frac{\dot{\mathbf{s}}_{ij}(t)}{||\dot{\mathbf{s}}_{ij}(t)||} = \frac{\dot{\mathbf{s}}_{ij}(t)}{a_{ij}(t)}, \label{eq:qhat}
    \end{align}
    where $\hat{p}_{ij}(t)\cdot\hat{q}_{ij}(t) = 0$ by \eqref{eq:tangency}, and both vectors are unit length. Thus, \eqref{eq:phat} and \eqref{eq:qhat} constitute an orthonormal basis for $\mathbb{R}^2$.
\end{definition}

Next, we find the projection of $\ddot{\mathbf{s}}_{ij}(t)$ onto the new contact basis. From \eqref{eq:tangency} we have
\begin{equation} \label{eq:sddotPhat}
    \ddot{\mathbf{s}}_{ij}(t)\cdot\hat{p}_{ij}(t) = \ddot{\mathbf{s}}_{ij}(t)\cdot\frac{\mathbf{s}_{ij}(t)}{2R} = \frac{-a_{ij}^2(t)}{2R}.
\end{equation}
We apply integration by parts to find the $\hat{q}_{ij}(t)$ component of $\ddot{\mathbf{s}}_{ij}(t)$. First,
\begin{align}
    \int\ddot{\mathbf{s}}_{ij}(t)\cdot\dot{\mathbf{s}}_{ij}(t)~dt = \dot{\mathbf{s}}_{ij}(t)\cdot\dot{\mathbf{s}}_{ij}(t) - \int\ddot{\mathbf{s}}_{ij}(t)\cdot\dot{\mathbf{s}}_{ij}(t)~dt,
\end{align}
which implies
\begin{equation} \label{eq:parts-2}
    \int\ddot{\mathbf{s}}_{ij}(t)\cdot\dot{\mathbf{s}}_{ij}(t)~dt = \frac{1}{2} \dot{\mathbf{s}}_{ij}(t)\cdot\dot{\mathbf{s}}_{ij}(t) = \frac{1}{2}a_{ij}^2(t).
\end{equation}
Taking a time derivative of \eqref{eq:parts-2} yields
\begin{equation} \label{eq:sddotDotSdot}
    \ddot{\mathbf{s}}_{ij}(t)\cdot\dot{\mathbf{s}}_{ij}(t) = a_{ij}(t)\dot{a}_{ij}(t).
\end{equation}

Next we present Theorem \ref{thm:aNeq0}, which gives the optimal trajectory for agent $i$ whenever $a_{ij}(t)=0$ over any nonzero interval while the safety constraint is active.

\begin{theorem} \label{thm:aNeq0}
For any agents $i\in\mathcal{A}$ and $j\in\mathcal{V}_i(t)$, if $a_{ij}(t) = 0$ over some nonzero interval $t\in[t_1, t_2]$, then the optimal trajectory for agent $i$ is to follow $\mathbf{u}_i(t) = \mathbf{u}_j(t)$ for all $t\in[t_1, t_2)$.
\end{theorem}

\begin{proof}
By definition we have $a_{ij}(t) = |\dot{\mathbf{s}}_{ij}(t)| = 0$. This implies $\dot{s}_{ij}(t) = 0$, and therefore $\mathbf{v}_j(t) = \mathbf{v}_{i}(t)$ for all $t\in[t_1, t_2]$. Thus $\mathbf{u}_i(t) = \mathbf{u}_j(t)$ for all $t\in[t_1, t_2)$.
\end{proof}

Thus, for any agent $i\in\mathcal{A}$ which has an active safety constraint with some agent $j\in\mathcal{V}_i(t)$, Theorem \ref{thm:aNeq0} provides the optimal control input for agent $i$ in the case that $a_{ij}(t) = 0$ over a nonzero time interval.
If $a_{ij}(t) = 0$ for at single instant $t\in\mathbb{R}$, then the optimal solution at that instant must enforce continuity of $a_{ij}(t)$ and the constraint $\mathbf{s}_{ij}(t)\cdot\mathbf{s}_{ij}(t) = 4R^2$.

Finally, we may project the dynamics of agent $i$ onto the basis given in Definition \ref{def:localBasis} and solve for the optimal trajectory when $a_{ij}(t) \neq 0$.
Next, we will use \eqref{eq:sddotPhat} and \eqref{eq:sddotDotSdot} to project $\ddot{\mathbf{s}}_{ij}(t)$ onto the contact basis,
    \begin{equation} \label{eq:sddotContactBasis}
        \ddot{\mathbf{s}}_{ij}(t) = 
        \begin{bmatrix}
        -\frac{a_{ij}^2(t)}{2R}\\
        \dot{a}_{ij}(t)
        \end{bmatrix}
        \cdot
        \begin{bmatrix}
        \hat{p}_{ij}(t)\\
         \hat{q}_{ij}(t)
        \end{bmatrix}
        ,
    \end{equation}
    which we use to solve for the time derivatives of \eqref{eq:phat} and \eqref{eq:qhat}. First,
    \begin{align}
        \frac{d}{dt}\hat{p}_{ij}(t) = \frac{\dot{\mathbf{s}}_{ij}(t)}{2R} = \frac{a(t)}{2R}\hat{q}_{ij}(t). \label{eq:dPhatDt}
    \end{align}
    Then, by the quotient rule,
        \begin{align}
        \frac{d}{dt}\hat{q}_{ij}(t) &= \frac{\ddot{\mathbf{s}}_{ij}(t)~a_{ij}(t) - \dot{\mathbf{s}}_{ij}(t)~\dot{a}_{ij}(t)}{a_{ij}^2(t)} \notag \\
        &= \frac{a_{ij}(t)\big(-a^2(t)\frac{1}{2R}\hat{p}_{ij}(t) + \dot{a}_{ij}(t)\hat{q}_{ij}(t)\big)}{a_{ij}^2(t)} \notag\\
        &\hspace{1em}- \frac{\dot{\mathbf{s}}_{ij}(t)~\dot{a}_{ij}(t)}{a_{ij}^2(t)} \notag \\
        &= -\frac{a_{ij}(t)}{2R}\hat{p}_{ij}(t). \label{eq:dQhatDt}
    \end{align}

	From \eqref{eq:s}, we may now write $\ddot{\mathbf{s}}_{ij}(t)$ projected on to the contact basis (Definition \ref{def:localBasis}) as
	\begin{align}
	\ddot{\mathbf{s}}_{ij}(t) =&~ \mathbf{u}_j(t) + \boldsymbol{\lambda}_{i}^v(t) + \mu_i(t)\mathbf{s}_{ij}(t) \notag\\
	=&~\Big(\mathbf{u}_j(t) + \boldsymbol{\lambda}_i^v(t)\Big)
	\begin{bmatrix}
	\hat{p}_{ij}(t) \\ \hat{q}_{ij}(t)
	\end{bmatrix}
	+ \mu_i(t)
	\begin{bmatrix}
	2R \\ 0
	\end{bmatrix}. \label{eq:sddotExpanded}
	\end{align}
	Next, we set \eqref{eq:sddotContactBasis} equal to \eqref{eq:sddotExpanded} and rewrite it as a system of scalar equations,
	\begin{align}
	   \boldsymbol{\lambda}_i^v(t)\cdot\hat{p}_{ij}(t) &= -\frac{a_{ij}^2(t)}{2R} - 2R ~\mu_{ij}(t) - \mathbf{u}_j(t)\cdot\hat{p}_{ij}(t), \label{eq:lvp}\\
	    \boldsymbol{\lambda}_i^v(t)\cdot\hat{q}_{ij}(t) &= \dot{a}_{ij}(t) - \mathbf{u}_j(t)\cdot\hat{q}_{ij}(t). \label{eq:lvq}
	\end{align}
	Taking the time derivative of \eqref{eq:lvp} yieds
	\begin{align}
	\frac{a_{ij}(t)}{2R}\boldsymbol{\lambda}_i^v(t)\cdot\hat{q}_{ij}(t) + \dot{\boldsymbol{\lambda}}_i^v(t)\cdot\hat{p}_{ij}(t) = - \frac{a_{ij}(t)\dot{a}_{ij}(t)}{R} \notag\\ -2R\dot{\mu}_{ij}(t) - \dot{\mathbf{u}}_j(t)\cdot\hat{p}_{ij}(t) - \frac{a_{ij}(t)}{2R}\mathbf{u}_j(t)\cdot\hat{q}_{ij}(t).\label{eq:lvpDeriv}
	\end{align}
	We then substitute \eqref{eq:lvdot2} and \eqref{eq:lvq} into \eqref{eq:lvpDeriv}, which yields
	\begin{align}\label{eq:lpp}
	    \boldsymbol{\lambda}_i^p(t)\cdot\hat{p}_{ij}(t) &= \frac{3a_{ij}(t)\dot{a}_{ij}(t)}{2R} + 2R\dot{\mu}_{ij}(t) \notag\\&~+ \dot{\mathbf{u}}_{j}(t)\cdot\hat{p}_{ij}(t).
	\end{align}
	
	Taking a time derivative of \eqref{eq:lvq} yields,
	\begin{align}
	-\frac{a_{ij}(t)}{2R} \boldsymbol{\lambda}_i^v(t)\cdot\hat{p}_{ij}(t) + \dot{\boldsymbol{\lambda}}_i^v(t)\cdot\hat{q}_{ij}(t) = \ddot{a}_{ij}(t)\notag\\ ~- \dot{\mathbf{u}}_j(t)\cdot\hat{q}_{ij}(t) + \frac{a_{ij}(t)}{2R}\mathbf{u}_j(t)\cdot\hat{p}_{ij}(t), \label{eq:lvqDeriv}
    \end{align}
	and substituting \eqref{eq:lvdot2} and \eqref{eq:lvp} into \eqref{eq:lvqDeriv}, yields
	\begin{equation} \label{eq:lpq}
	    \boldsymbol{\lambda}_i^p\cdot\hat{q}_{ij}(t) = \dot{\mathbf{u}}_j(t)\cdot\hat{q}_{ij}(t) - \ddot{a}_{ij}(t) + \frac{a_{ij}^3(t)}{4R^2}.
	\end{equation}
	
	We then take an additional time derivative of \eqref{eq:lpp} and \eqref{eq:lpq}.  This yields
	\begin{align}
	    \dot{\boldsymbol{\lambda}}_i^p(t)\cdot\hat{p}_{ij}(t) &= - \frac{a_{ij}(t)}{2R}\boldsymbol{\lambda}_i^p(t)\cdot\hat{q}_{ij}(t) \notag\\ &~+  \frac{3}{2R}(\dot{a}_{ij}^2(t) + a_{ij}(t)\ddot{a}_{ij}(t)) + 2R\dot{\mu}_{ij}(t) \notag\\
	    &~+ \frac{a_{ij}(t)}{2R}\dot{\mathbf{u}}_{j}(t)\cdot\hat{q}_{ij}(t)
	    ~+ \ddot{\mathbf{u}}_j(t) \hat{p}(t), \label{eq:lppdot}
	    \\
	    \dot{\boldsymbol{\lambda}}_i^p(t)\cdot\hat{q}_{ij}(t) &= \frac{a_{ij}(t)}{2R}\boldsymbol{\lambda}_i^p(t)\cdot\hat{p}_{ij}(t) \notag\\&~
	     - \dddot{a}_{ij}(t) + \frac{3}{4R^2} a_{ij}^2(t)\dot{a}_{ij}(t)\notag\\
	     &~- \frac{a_{ij}(t)}{2R}\dot{\mathbf{u}}_j(t)\cdot\hat{p}_{ij}(t) + \ddot{\mathbf{u}}_j(t)\cdot\hat{q}(t).\label{eq:lpqdot}
	\end{align}
	Substituting \eqref{eq:lpdot2}, \eqref{eq:lpp}, and \eqref{eq:lpq} into \eqref{eq:lppdot} and \eqref{eq:lpqdot} yields a system of nonlinear ordinary differential equations,
	\begin{align}
    \frac{a_{ij}^2(t)}{2R} \mu_{ij}(t) +  
	\frac{a_{ij}^4(t)}{8R^3} &=
	2R\dot{\mu}_{ij}(t) +
	\frac{4}{2R}a_{ij}(t)\ddot{a}_{ij}(t) \notag\\ &~~+
	\frac{3}{2R}\dot{a}_{ij}(t) +
	\ddot{\mathbf{u}}_j(t)\cdot\hat{p}_{ij}(t),
	\label{eq:diffeq1}  \\
	a_{ij}(t)\dot{\mu}_{ij}(t) &+
	\dot{a}_{ij}(t)\mu_{ij}(t) + 
	\ddot{\mathbf{u}}_j(t)\cdot\hat{q}_{ij}(t) \notag\\ &~~+
	\frac{6}{4R^2}a_{ij}^2(t)\dot{a}_{ij}(t) =
	\dddot{a}_{ij}(t).
	\label{eq:diffeq2}
	\end{align}

    Thus, for any constrained trajectory to be energy-optimal, it must be a solution of \eqref{eq:diffeq1} and \eqref{eq:diffeq2}.
    In general, finding a solution is rare, since both equations are nonlinear and \eqref{eq:diffeq2} is third order.
    Therefore, our approach will be to impose $\dot{a}_{ij}(t) = 0$ over the constraint arc, which is a locally optimal trajectory per Theorem \ref{thm:aNeq0}.
    Thus, the remaining unknown quantities are: the junction time when agent $i$ transitions from the unconstrained to constrained arc, $t_1$, the junction time that agent $i$ exits from the constrained arc, $t_2$, and the initial 
    orientation of the 
    vector $\mathbf{s}_{ij}(t_1)$.
    These quantities are coupled to the proceeding and following unconstrained arcs by the jump conditions \cite{Bryson1975AppliedControl},
    \begin{align}
        \mathbf{x}_i(t_1^-) &= \mathbf{x}_i(t_1^+), \label{eq:jump1}\\
        \boldsymbol{\lambda}_i^T(t_1^-) &= \boldsymbol{\lambda}_i^T(t_1^+) + \boldsymbol{\nu}_i^T\frac{\partial N_i\big(t, \mathbf{x}_i(t)\big)}{\partial \mathbf{x}(t)}\Bigg|_{t=t_1}, \label{eq:jump2} \\
        H(t_1^-) &= H(t_1^+) + \boldsymbol{\nu}_i^T \frac{\partial N_i\big(t, \mathbf{x}_i(t)\big)}{\partial t}\Bigg|_{t=t_1}, \label{eq:jump3} \\
        H(t_2^-) &= H(t_2^+), \label{eq:jump4} \\
        \mathbf{x}_i(t_2^-) &= \mathbf{x}_i(t_2^+), \label{eq:jump5} \\
        \boldsymbol{\lambda}_i(t_2^-) &= \boldsymbol{\lambda}_i(t_2^+), \label{eq:jump6}
    \end{align}
    where the superscripts $t^-$ and $t^+$ correspond to the left and right-side limits of $t$, respectively. Thus, $t_1^-$ and $t_2^+$ correspond to the unconstrained arcs at the junctions $t_1$ and $t_2$. Likewise, $t_1^+$ and $t_2^-$ correspond to the junctions where agent $i$ enters and exits the constrained arc, respectively. The constant vector $\boldsymbol{\nu}_i$ is given by \cite{Bryson1975AppliedControl}
    \begin{equation}
        \boldsymbol{\nu}_i = 
        \begin{bmatrix}
        -\frac{\mathbf{u}_i(t_1^-)\cdot \mathbf{u}_i(t_1^-)}{2(\mathbf{s}_{ij}(t_1^-)\cdot \mathbf{v}_i(t_1^-))} \\
        0
        \end{bmatrix}.
    \end{equation}
    
    For the case when agent $i$'s trajectory has only a single constrained arc, \eqref{eq:jump1} - \eqref{eq:jump6} coupled with the initial and final conditions, \eqref{eqn:BCs} and \eqref{eqn:ICs}, constitute 26 scalar equations to solve for the 26 unknowns (8+8+8 constants of integration + 2 transition times). When additional constrained arcs become active, additional jump conditions must be computed using \eqref{eq:jump1} - \eqref{eq:jump6}. The entire system of equations is then solved simultaneously to yield the energy-optimal trajectory for agent $i$.

	\subsection{The Full Solution to Problem \ref{prb:trajectory}}
	
	So far, we have provided the unconstrained and safety-constrained arcs with a relaxation of the state and control constraints. An extension to the fully-constrained case for agent $i\in\mathcal{A}$ is straightforward, and the solution for every possible case is outlined below.
	\begin{enumerate}
	    \item No constraints are active: Agent $i$ will follow an unconstrained trajectory.
	    \item Only one safety constraint is active: Agent $i$ will follow the trajectory outlined in Section \ref{sec:constrainedSolution}.
	    \item More than one safety constraint is active: By Lemma \ref{lma:manyConstraints}, the unique trajectory of agent $i$ will be defined by the active constraints.
	    \item Only one state/control constraint is active: This reduces to a steering problem, where agent $i$ follows a known velocity profile and must arrive at a target state along a minimum-energy path \cite{Ross2015}.
	    \item One safety and one state/control constraint are active: Agent $i$ must follow the path imposed by the safety constraint with a speed profile determined by the state/control constraint. The unique solution to this problem is, therefore, optimal.
	\end{enumerate}
	The state trajectory of agent $i\in\mathcal{A}$ must be a piecewise-continuous function consisting of the five possible cases. These segments are then pieced together using the optimality conditions. We presented the optimality conditions for collision avoidance in Section \ref{sec:constrainedSolution}; the conditions for the state and control constraints are derived in \cite{Malikopoulos2018} and \cite{Bryson1975AppliedControl}.

\section{Simulation Results}\label{sec:simulation}
To provide insight into the behavior of the agents, a series of simulations were performed with $M=N=10$ agents and a time parameter of $T=10$ s. The simulations were run for $t=20$ s or until all agents reach their assigned goal, whichever occurred later. The center of the formation moved with a velocity of $\mathbf{v}_{cg}=[0.15, ~ 0.35]$ m/s; the leftmost and rightmost three goals each move with an additional periodic velocity of $[0.125\cos{0.75t}, ~ 0]$ m/s relative to the formation. Videos of the simulation results can be found at  \url{https://sites.google.com/view/ud-ids-lab/omas}.
 
 \begin{figure}[ht]
    \centering
    \includegraphics[width=0.85\linewidth]{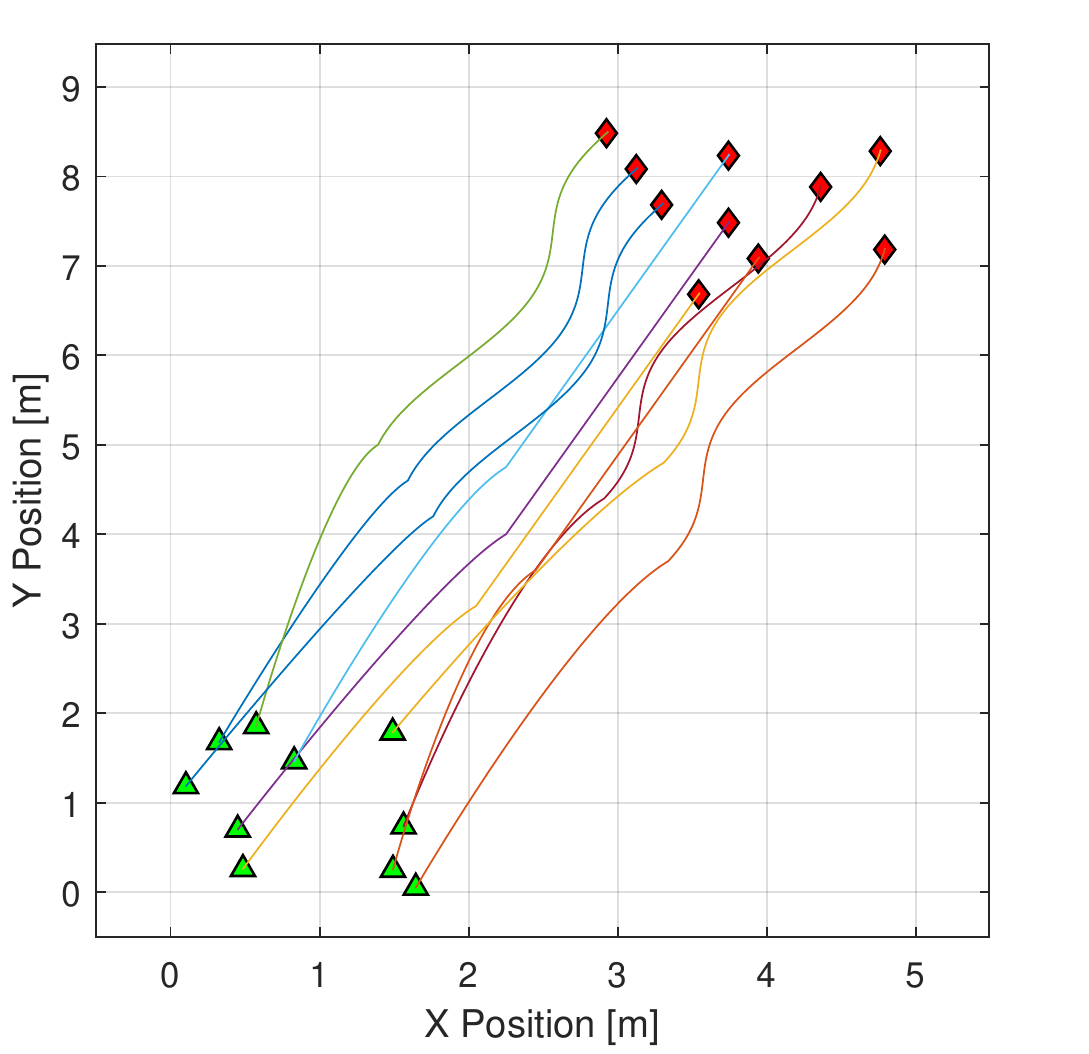}
    \caption{Simulation result for the centralized case. Goals which minimize the unconstrained trajectories are assigned to the agents once at $t_i^0$.}
    \label{fig:h=inf}
\end{figure}
 
The minimum separating distance between agents, total energy consumed, and maximum velocity for the unconstrained solutions to Problem \ref{prb:trajectory} are all given as a function of the horizon in Table \ref{tab:resultsEnergy}. The energy consumption only considers the energy required to reach the goal, which, in this case, was significantly lower than the energy required to maintain the formation. The trajectory of each agent over time is given in Figures \ref{fig:h=inf}--\ref{fig:h=0.75} for varying sensing horizon values. Although the trajectories may appear to cross in Figures \ref{fig:h=inf}--\ref{fig:h=0.75}, they are only crossing in space and not in time.
 
\begin{table}[ht]
    \centering
    \caption{Numerical results for N=10 agents and M=10 goals for various sensing distances.}
    \label{tab:resultsEnergy}
    \begin{tabular}{ccccc}
    $h$ [m] & min. separation & energy & $t_f$ & Total Bans \\
                & [cm]      &    [J/kg] & [s]  & \\ \toprule
    $\infty$ &	25.25&	0.85&	20& 	0 \\
    1.60&	1.64&	1.10&	20& 	4 \\
    1.50&	1.60&	1.17&	20&  	24\\
    1.40&	2.01&	1.96&	23.3&	31\\
    1.30&	0.33&	1670&	26.05&	36\\
    1.20&   0.65&    866&   25.35&  34\\
    1.10&	1.05&	5370&	26.85&	40\\
    1.00&   1.96&   7609&   30.65&  35\\
    0.95&	3.12&	3149&	25.05&	27\\
    0.75&	1.37&	6.87&	20&	    35\\
    0.50&   0.27&	692.0&	26.65&	35
    \end{tabular}
\end{table}

The performance of our algorithm is strongly affected by how much information is available to each agent. This is a function of the sensing horizon, initial states of the agents, and the desired formation shape. Generally, better overall performance requires the agents to have more information. However, it is not apparent what information is necessary; in fact, the results in Table \ref{tab:resultsEnergy} generally show no correlation between energy consumption and sensing horizon.

The trade-off for more information is in the computational and sensing load imposed on each agent. As an agent observes more of the system (via sensing, communication, or memory), the computational burden to solve the assignment and trajectory generation problems also increases. However, this computational cost does not necessarily result in improved system performance, as demonstrated in Table \ref{tab:resultsEnergy}.

\begin{figure}[ht]
    \centering
    \includegraphics[width=0.75\linewidth]{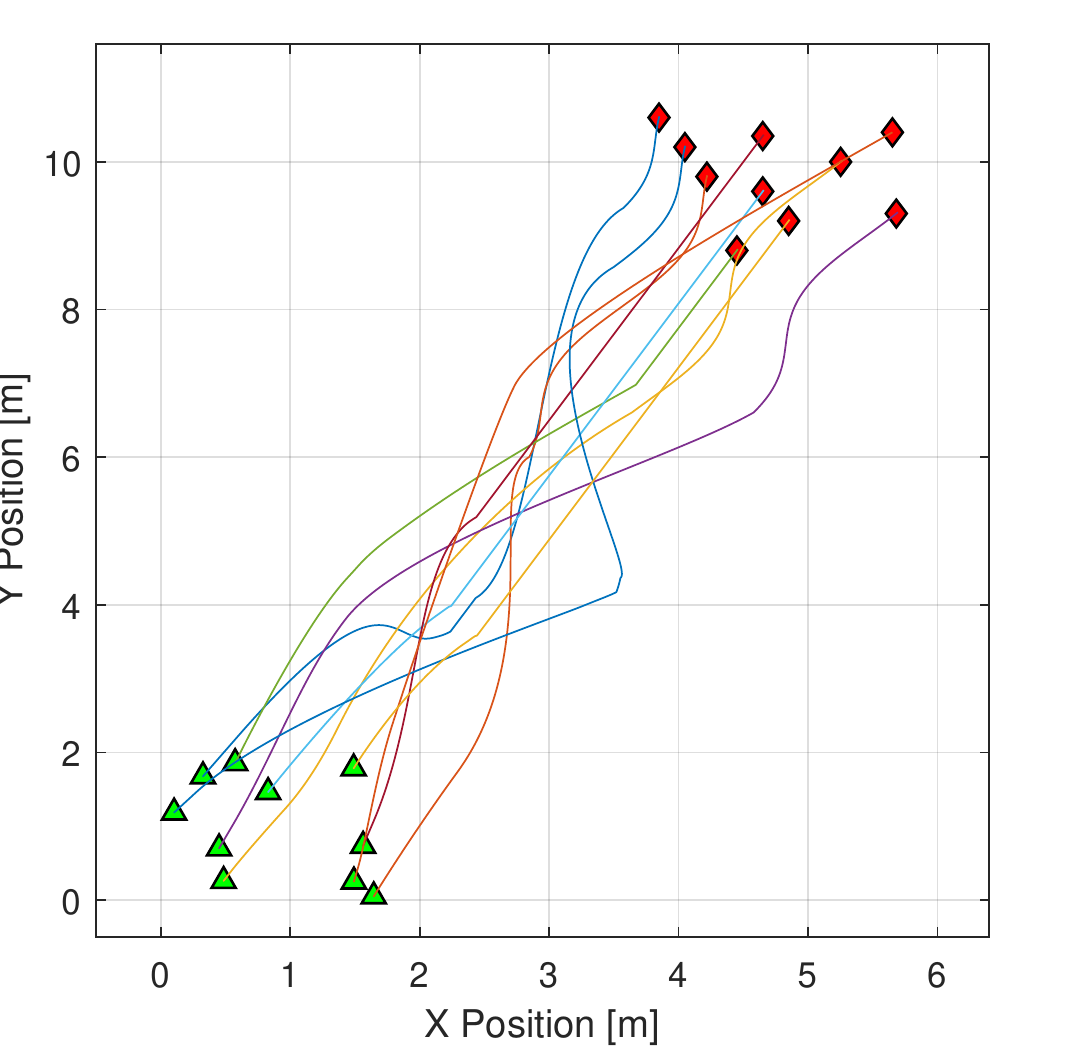}
    \caption{Simulation result for $h=1.30$ m. The agents do not start with a globally unique assignment, and several agents must re-route partway through the simulation. Although the trajectories cross in space they do not cross in time.}
    \label{fig:h=130}
\end{figure}

\begin{figure}[ht]
    \centering
    \includegraphics[width=0.75\linewidth]{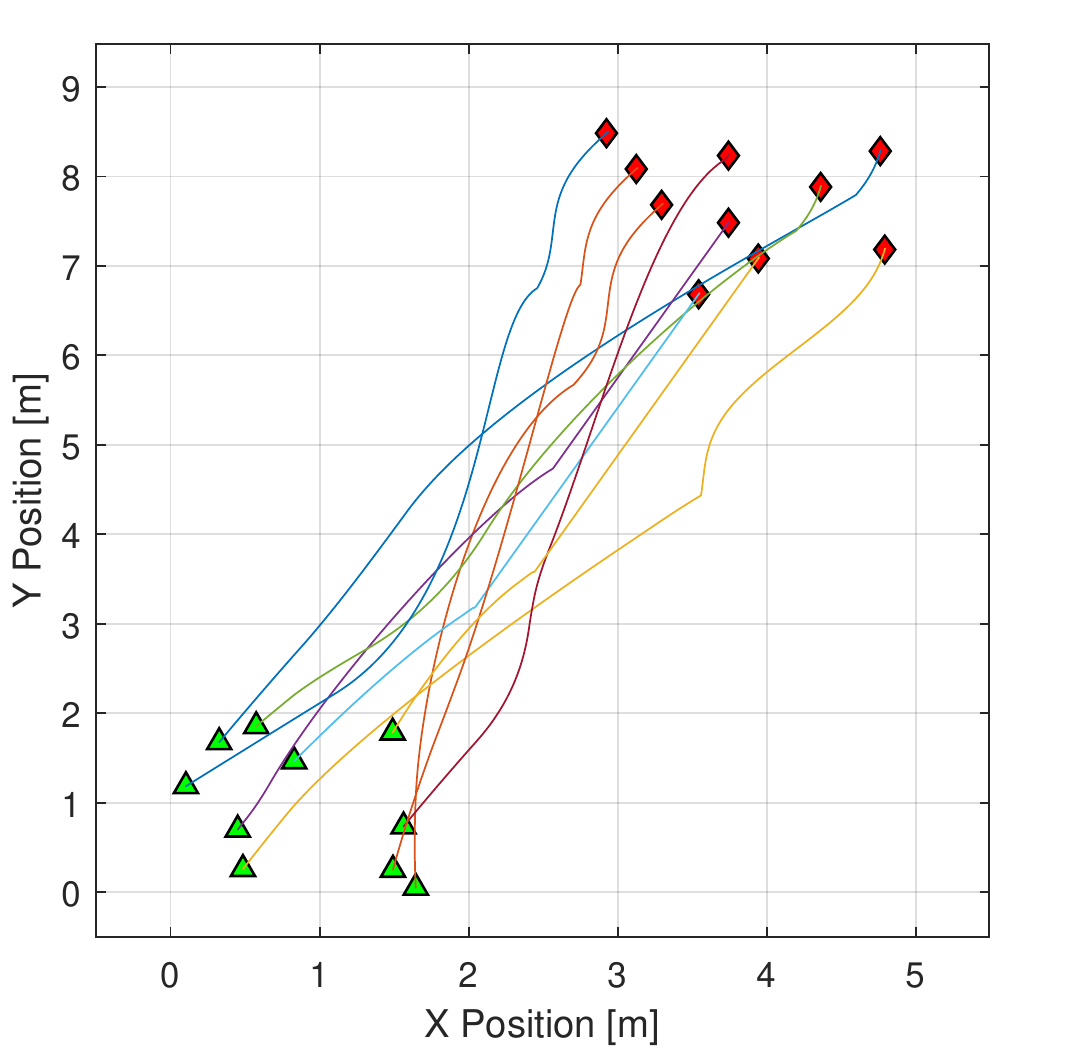}
    \caption{Simulation result for $h=0.75$ m. Although the horizon for this case is smaller than in Figure \ref{fig:h=130}, the system dynamics happen to result in more efficient trajectories overall.}
    \label{fig:h=0.75}
\end{figure}

\section{Conclusion}\label{sec:conclusion}

In this paper, we proposed a decentralized framework for moving a group of autonomous agents into a desired formation. The only information required a priori is the positions of the goals in a global coordinate frame. We provided guarantees of convergence under Assumption \ref{smp:bans}. Our method leveraged a set of agent interaction dynamics, which allowed a decentralized calculation of the priority order for agents.
We also derived local energy-optimal trajectories for constrained and unconstrained paths and presented the conditions for optimality in the form of a boundary-value nonlinear ordinary differential equation. The resulting optimal controller was validated in MATLAB using ten agents and ten dynamic moving goals with varying values of the sensing radius.

One area for future research is finding a relaxation of Assumption \ref{smp:bans} or another fundamental condition to replace it. Deriving additional locally-optimal solutions to \eqref{eq:diffeq1} and \eqref{eq:diffeq2} is another research direction. A relaxation of the jump conditions to find approximately optimal trajectories that can be generated in real-time is another area which is under active research \cite{Beaver2020AnFlocking}.
Finally, analysis of the system parameter $T$ using the fully constrained agent trajectories is another potential direction for future research. This includes methods to optimally select $T$ or to estimate its magnitude based on the state and observations of each agent.

\bibliography{MendeleyBlue,Andreas}

\end{document}